\documentclass[oneside,reqno,english]{amsart}
\usepackage[T1]{fontenc}
\usepackage[utf8]{inputenc}
\usepackage{xcolor}
\usepackage{babel}
\usepackage{prettyref}
\usepackage{amstext}
\usepackage{amsthm}
\usepackage{amssymb}
\usepackage[pdfusetitle,
 bookmarks=true,bookmarksnumbered=false,bookmarksopen=false,
 breaklinks=false,pdfborder={0 0 0},pdfborderstyle={},backref=false,colorlinks=false]
 {hyperref}
\hypersetup{
 colorlinks=true,citecolor=blue,linkcolor=blue,linktocpage=true}

\makeatletter
\numberwithin{equation}{section}
\numberwithin{figure}{section}

\usepackage{prettyref}

\newrefformat{cor}{Corollary~\ref{#1}}
\newrefformat{subsec}{Section~\ref{#1}}
\newrefformat{lem}{Lemma~\ref{#1}}
\newrefformat{thm}{Theorem~\ref{#1}}
\newrefformat{sec}{Section~\ref{#1}}
\newrefformat{chap}{Chapter~\ref{#1}}
\newrefformat{prop}{Proposition~\ref{#1}}
\newrefformat{exa}{Example~\ref{#1}}
\newrefformat{tab}{Table~\ref{#1}}
\newrefformat{rem}{Remark~\ref{#1}}
\newrefformat{def}{Definition~\ref{#1}}
\newrefformat{fig}{Figure~\ref{#1}}
\newrefformat{claim}{Claim~\ref{#1}}
\newrefformat{assu}{Assumption~\ref{#1}}

\makeatother

\theoremstyle{definition}
\newtheorem*{problem*}{\protect\problemname}
\theoremstyle{plain}
\newtheorem{thm}{\protect\theoremname}[section]
\theoremstyle{definition}
\newtheorem{defn}[thm]{\protect\definitionname}
\theoremstyle{plain}
\newtheorem{lem}[thm]{\protect\lemmaname}
\newtheorem{cor}[thm]{\protect\corollaryname}
\theoremstyle{remark}
\newtheorem{rem}[thm]{\protect\remarkname}
\theoremstyle{definition}
\newtheorem{example}[thm]{\protect\examplename}
\providecommand{\corollaryname}{Corollary}
\providecommand{\definitionname}{Definition}
\providecommand{\examplename}{Example}
\providecommand{\lemmaname}{Lemma}
\providecommand{\problemname}{Problem}
\providecommand{\remarkname}{Remark}
\providecommand{\theoremname}{Theorem}

\begin{document}
\title{Extensions of Operator-Valued Kernels on $\mathbb{F}^{+}_{d}$ }
\begin{abstract}
We study the problem of extending a positive definite operator-valued
kernel, defined on words of a fixed finite length from a free semigroup,
to a global kernel defined on all words. We show that if the initial
kernel satisfies a one-step dominance inequality on its interior,
a global extension that preserves this interior data and the dominance
property is always possible. For the problem of matching the kernel
on the boundary, we introduce a shift-consistency condition. We prove
this condition is sufficient to guarantee the existence of a global
extension that agrees with the original kernel on its entire domain.
\end{abstract}

\author{James Tian}
\address{Mathematical Reviews, 535 W. William St, Suite 210, Ann Arbor, MI
48103, USA}
\email{james.ftian@gmail.com}

\maketitle
\tableofcontents{}

\section{Introduction}\label{sec:1}

This paper begins with a simple extension question. Suppose that we
are given a positive definite kernel on the words of length at most
$N$ in a free semigroup. When can this finite piece be continued
to a positive definite kernel on all words?

There is one additional feature that we want the extension to preserve.
Let $\mathbb{F}^{+}_{d}$ be the free semigroup on $d$ generators,
and write $\Lambda_{N}=\left\{ \alpha\in\mathbb{F}^{+}_{d}:\left|\alpha\right|\leq N\right\} $.
Given a kernel $K:\Lambda_{N}\times\Lambda_{N}\rightarrow L\left(H\right)$,
define $K_{\Sigma}\left(\alpha,\beta\right)=\sum^{d}_{i=1}K\left(\alpha i,\beta i\right)$.
We assume that $K$ is positive definite and that $K_{\Sigma}\leq K$
on $\Lambda_{N-1}\times\Lambda_{N-1}$. The question is whether there
is a global positive definite kernel $\widetilde{K}:\mathbb{F}^{+}_{d}\times\mathbb{F}^{+}_{d}\rightarrow L\left(H\right)$
that still satisfies $\widetilde{K}_{\Sigma}\leq\widetilde{K}$. 

Initially one might simply ask that $\widetilde{K}$ agree with $K$
everywhere on $\Lambda_{N}\times\Lambda_{N}$. But it is useful to
separate two questions. 

(1) Can we preserve the data on the interior $\Lambda_{N-1}\times\Lambda_{N-1}$?

(2) And can we also preserve the entries involving words of length
$N$?

The first question has a clean answer. The one-step inequality $K_{\Sigma}\leq K$
is enough. In \prettyref{thm:3-1}, we show that there is always a
global positive definite kernel $\widetilde{K}$ such that 
\[
\widetilde{K}\left(\alpha,\beta\right)=K\left(\alpha,\beta\right),\quad\alpha,\beta\in\Lambda_{N-1},
\]
and $\widetilde{K}_{\Sigma}\leq\widetilde{K}$ on the whole free semigroup.
The extension also has a concrete operator model. It can be written
as 
\[
\widetilde{K}\left(\alpha,\beta\right)=W^{*}S^{\alpha}P\left(S^{\beta}\right)^{*}W
\]
where $S=\left(S_{1},\ldots,S_{d}\right)$ is a row isometry with
orthogonal ranges, $P$ is an orthogonal projection, and $W:H\rightarrow\mathcal{K}$
is bounded. 

The boundary question (2) is different. The values of $K$ involving
words of length $N$ prescribe more than the shift on the interior.
They prescribe how the new vectors pair with the interior vectors,
and they also prescribe how the new vectors pair with one another.
These requirements need not be compatible with operators acting on
the interior Kolmogorov space.

To handle this, we introduce a shift-consistency condition in \prettyref{def:2-5}.
It asks that the level-$N$ data be produced by a column contraction
extending the shifts that are already visible on shorter words. When
this condition holds, \prettyref{thm:b8} gives a global kernel satisfying
\[
\widetilde{K}\left(\alpha,\beta\right)=K\left(\alpha,\beta\right),\quad\alpha,\beta\in\Lambda_{N}.
\]
We use shift consistency as a sufficient condition for boundary agreement,
and we do not claim that it is necessary.

Positive definite kernel extensions have long been part of interpolation,
moment theory, and operator realization \cite{MR51437}. In one variable,
the surrounding ideas include Szegő-Toeplitz theory, Pick-Nevanlinna
interpolation, and the classical moment problems \cite{Toeplitz1910,MR77,MR2760647,MR1511844,Nevanlinna:1919aa,MR208383,MR1882259,MR1544592,MR4191205,MR1303090,MR2105088,MR2105089,MR2589247}.
On free semigroups, the natural background comes from dilation and
model theory for row contractions, free semigroup operator algebras,
and noncommutative reproducing kernel spaces \cite{MR671311,MR744917,MR972704,MR1625750,MR1681749,MR2752983,MR3244229,MR3526117,MR3792241,MR4744483}.

There are also noncommutative moment problems in which one prescribes
values of a completely positive map on part of a free or product semigroup
and asks for a representation of the associated operator algebra \cite{MR1618326}.
The data in the present paper have a different form. We begin with
a positive definite kernel on a finite ball and ask to preserve the
inequality $K_{\Sigma}\leq K$. This makes the distinction between
the interior and the boundary part of the problem from the beginning.

\section{Local extension on the free semigroup}\label{sec:2}

Fix $d\in\mathbb{N}$ and the free semigroup $\mathbb{F}^{+}_{d}$
on $\{1,\dots,d\}$ with neutral element $\emptyset$. For a word
$\alpha=i_{1}\cdots i_{k}\in\mathbb{F}^{+}_{d}$ write $|\alpha|=k$
and $\tilde{\alpha}=i_{k}\cdots i_{1}$. Concatenation is denoted
by juxtaposition, so $\alpha i:=i_{1}\cdots i_{k}i$.

For $N\in\mathbb{N}$, set 
\[
\Lambda_{N}:=\left\{ \alpha\in\mathbb{F}^{+}_{d}:\left|\alpha\right|\le N\right\} ,\quad\partial\Lambda_{N}:=\left\{ \alpha\in\mathbb{F}^{+}_{d}:\left|\alpha\right|=N\right\} .
\]
We refer to $\Lambda_{N-1}$ as the interior and $\partial\Lambda_{N}$
as the boundary at level $N$.

Let $H$ be a Hilbert space and $L\left(H\right)$ the space of bounded
linear operators on $H$. A map 
\[
K:E\times E\longrightarrow L(H)
\]
defined on a finite index set $E\subset\mathbb{F}^{+}_{d}$ is called
positive definite (p.d.) on $E$ if for every finite family $\left\{ \left(\alpha_{j},u_{j}\right)\right\} ^{m}_{j=1}\subset E\times H$,
\[
\sum^{m}_{j,k=1}\left\langle u_{j},K\left(\alpha_{j},\alpha_{k}\right)u_{k}\right\rangle _{H}\ge0.
\]
If $E=\Lambda_{N}$ we simply say “$K$ is p.d. on $\Lambda_{N}$.”

Whenever $\alpha i,\beta i\in E$ for all $i=1,\dots,d$, we define
the (one-step) shifted kernel 
\[
K_{\Sigma}\left(\alpha,\beta\right):=\sum^{d}_{i=1}K\left(\alpha i,\beta i\right).
\]
When both sides make sense on a set $F\subset E$, we write $K_{\Sigma}\le K$
on $F$ to mean that $K-K_{\Sigma}$ is p.d. on $F$. 
\begin{problem*}
Given a level $N$ and a p.d. kernel 
\[
K:\Lambda_{N}\times\Lambda_{N}\to L\left(H\right),
\]
decide when there exists a global p.d. kernel $\widetilde{K}:\mathbb{F}^{+}_{d}\times\mathbb{F}^{+}_{d}\to L\left(H\right)$
such that: 
\begin{itemize}
\item[E1.] 
\[
\widetilde{K}\left(\alpha,\beta\right)=K\left(\alpha,\beta\right),\quad\alpha,\beta\in\Lambda_{N-1};
\]
\item[E2.] 
\[
\widetilde{K}_{\Sigma}\le\widetilde{K}\quad\text{on }\mathbb{F}^{+}_{d}\times\mathbb{F}^{+}_{d};
\]
\item[E3.] 
\[
\widetilde{K}\left(\alpha,\beta\right)=K\left(\alpha,\beta\right),\quad\alpha,\beta\in\Lambda_{N}.
\]
\end{itemize}
\end{problem*}
We will separate the existence of an extension satisfying (E1) \&
(E2) from the sharper boundary agreement question (E3).

For later use, we single out two  families of admissible completions. 
\begin{defn}
Fix $N$ and a p.d. $K$ on $\Lambda_{N}$. 
\begin{enumerate}
\item $E_{int}\left(K\right)$ denotes the set of all p.d. kernels $\widetilde{K}$
on $\mathbb{F}^{+}_{d}$ such that (E1) and (E2) hold. 
\item $E_{bd}(K)$ denotes the subset of $E_{int}\left(K\right)$ consisting
of those $\widetilde{K}$ that also satisfy (E3) (full agreement on
$\Lambda_{N}$). 
\end{enumerate}
\end{defn}

To phrase the boundary question intrinsically on the given data, we
use the Kolmogorov (RKHS) space of the truncation and the compressed
shifts. 

Let $K:\Lambda_{N}\times\Lambda_{N}\rightarrow L\left(H\right)$ be
a p.d. kernel. Let $H^{\left(N\right)}_{K}$ be the associated Kolmogorov
space, with canonical vectors $V_{\alpha}u$ ($\alpha\in\Lambda_{N}$,
$u\in H$), so that 
\[
\left\langle V_{\alpha}u,V_{\beta}v\right\rangle _{H^{\left(N\right)}_{K}}=\left\langle u,K\left(\alpha,\beta\right)v\right\rangle _{H}
\]
for all $u,v\in H$, and all $\alpha,\beta\in\Lambda_{N}$. That is
(see, e.g., \cite{MR2938971,MR4250453,tian2025randomoperatorvaluedkernelsmoment}),
\[
K\left(\alpha,\beta\right)=V^{*}_{\alpha}V_{\beta},\quad\alpha,\beta\in\Lambda_{N}.
\]
Set 
\begin{align}
H^{\left(N-1\right)}_{K} & :=\overline{span}\left\{ V_{\alpha}u:\alpha\in\Lambda_{N-1},\ u\in H\right\} \subset H^{\left(N\right)}_{K},\label{eq:2-1}\\
H^{\left(N-2\right)}_{K} & :=\overline{span}\left\{ V_{\alpha}u:\alpha\in\Lambda_{N-2},\ u\in H\right\} \subset H^{\left(N-1\right)}_{K},\label{eq:2-2}
\end{align}
where, when $N=1$, we use the convention 
\[
\Lambda_{-1}=\emptyset,\qquad H^{\left(-1\right)}_{K}=\left\{ 0\right\} ,\qquad D_{-1}=\left\{ 0\right\} .
\]
Let 
\[
P_{N-2}:H^{\left(N-1\right)}_{K}\rightarrow H^{\left(N-2\right)}_{K}
\]
be the orthogonal projection.

On 
\[
D_{N-2}:=span\left\{ V_{\alpha}u:\alpha\in\Lambda_{N-2},\ u\in H\right\} ,
\]
define 
\[
C_{i}\left(\sum_{\alpha\in\Lambda_{N-2}}V_{\alpha}u_{\alpha}\right):=\sum_{\alpha\in\Lambda_{N-2}}V_{\alpha i}u_{\alpha},\quad i=1,\dots,d.
\]

If $K_{\Sigma}\le K$ on $\Lambda_{N-1}\times\Lambda_{N-1}$, the
column map 
\[
Cy:=\left(C_{1}y,\dots,C_{d}y\right),\quad D_{N-2}\rightarrow\left(H^{\left(N-1\right)}_{K}\right)^{\oplus d},
\]
is well defined and contractive, and hence extends by continuity to
$H^{\left(N-2\right)}_{K}$.

Define 
\[
B_{i}:=C_{i}P_{N-2},\quad i=1,\dots,d.
\]

\begin{lem}
\label{lem:2-2}Then the column map 
\[
Bx:=\left(B_{1}x,\dots,B_{d}x\right),\quad H^{\left(N-1\right)}_{K}\rightarrow\left(H^{\left(N-1\right)}_{K}\right)^{\oplus d},
\]
is contractive. In particular, 
\[
\sum\nolimits_{i}B^{*}_{i}B_{i}\le I_{H^{\left(N-1\right)}_{K}}.
\]
Moreover, 
\[
B_{i}V_{\alpha}u=V_{\alpha i}u
\]
for every $\alpha\in\Lambda_{N-2}$, $u\in H$, and $i=1,\dots,d$. 
\end{lem}

\begin{proof}
Let 
\[
y=\sum_{\alpha\in\Lambda_{N-2}}V_{\alpha}u_{\alpha}\in D_{N-2}.
\]
Using the RKHS inner product, 
\begin{align}
\sum\nolimits^{d}_{i=1}\left\Vert C_{i}y\right\Vert ^{2}_{H^{\left(N-1\right)}_{K}} & =\sum_{\alpha,\beta\in\Lambda_{N-2}}\sum\nolimits^{d}_{i=1}\left\langle u_{\alpha},K\left(\alpha i,\beta i\right)u_{\beta}\right\rangle _{H}\nonumber \\
 & =\sum_{\alpha,\beta\in\Lambda_{N-2}}\left\langle u_{\alpha},K_{\Sigma}\left(\alpha,\beta\right)u_{\beta}\right\rangle _{H}\nonumber \\
 & \le\sum_{\alpha,\beta\in\Lambda_{N-2}}\left\langle u_{\alpha},K\left(\alpha,\beta\right)u_{\beta}\right\rangle _{H}=\left\Vert y\right\Vert ^{2}_{H^{\left(N-2\right)}_{K}}.\label{eq:b1}
\end{align}

In particular, if $y=\sum_{\alpha\in\Lambda_{N-2}}V_{\alpha}u_{\alpha}=0$,
then $\sum\nolimits^{d}_{i=1}\left\Vert C_{i}y\right\Vert ^{2}=0$.
Thus each $C_{i}$ is well defined on $D_{N-2}$, and \eqref{eq:b1}
shows that the column map $C$ is contractive. It therefore extends
by continuity to $H^{\left(N-2\right)}_{K}$.

Now let $x\in H^{\left(N-1\right)}_{K}$. Since $B_{i}=C_{i}P_{N-2}$,
\[
\sum\nolimits^{d}_{i=1}\left\Vert B_{i}x\right\Vert ^{2}=\sum\nolimits^{d}_{i=1}\left\Vert C_{i}P_{N-2}x\right\Vert ^{2}\le\left\Vert P_{N-2}x\right\Vert ^{2}\le\left\Vert x\right\Vert ^{2}.
\]
Hence $B$ is contractive and 
\[
\sum\nolimits_{i}B^{*}_{i}B_{i}\le I_{H^{\left(N-1\right)}_{K}}.
\]

Finally, if $\alpha\in\Lambda_{N-2}$, then $V_{\alpha}u\in H^{\left(N-2\right)}_{K}$,
and therefore 
\[
B_{i}V_{\alpha}u=C_{i}P_{N-2}V_{\alpha}u=C_{i}V_{\alpha}u=V_{\alpha i}u.
\]
\end{proof}

\begin{cor}
\label{cor:2-3} Assume $K$ is p.d. on $\Lambda_{N}$ and $K_{\Sigma}\le K$
on $\Lambda_{N-1}\times\Lambda_{N-1}$. With $H^{\left(N-1\right)}_{K}$
and the operators $B_{1},\dots,B_{d}$ as in \prettyref{lem:2-2},
set 
\[
A_{N-1}:=\sum^{d}_{i=1}B^{*}_{i}B_{i}\in\mathcal{L}(H^{\left(N-1\right)}_{K}),\quad0\le A_{N-1}\le I_{H^{\left(N-1\right)}_{K}}.
\]
Then for all $\alpha,\beta\in\Lambda_{N-2}$, 
\[
K_{\Sigma}(\alpha,\beta)\ =\ V^{*}_{\alpha}\,A_{N-1}\,V_{\beta}.
\]
\end{cor}

\begin{proof}
Let $y=\sum_{\alpha\in\Lambda_{N-2}}V_{\alpha}u_{\alpha}\in H^{\left(N-2\right)}_{K}\subset H^{\left(N-1\right)}_{K}$.
Then 
\[
\sum^{d}_{i=1}\left\Vert B_{i}y\right\Vert ^{2}_{H^{\left(N\right)}_{K}}=\sum_{\alpha,\beta\in\Lambda_{N-2}}\sum^{d}_{i=1}\left\langle u_{\alpha},K\left(\alpha i,\beta i\right)u_{\beta}\right\rangle _{H}=\sum_{\alpha,\beta\in\Lambda_{N-2}}\left\langle u_{\alpha},K_{\Sigma}\left(\alpha,\beta\right)u_{\beta}\right\rangle _{H},
\]
while also 
\[
\sum^{d}_{i=1}\left\Vert B_{i}y\right\Vert ^{2}_{H^{\left(N\right)}_{K}}=\left\langle y,A_{N-1}y\right\rangle _{H^{\left(N\right)}_{K}}=\sum_{\alpha,\beta\in\Lambda_{N-2}}\left\langle u_{\alpha},V^{*}_{\alpha}A_{N-1}V_{\beta}u_{\beta}\right\rangle _{H}.
\]
Identifying the two quadratic forms on the dense span of $\left\{ V_{\alpha}u:\alpha\in\Lambda_{N-2}\right\} $
gives $K_{\Sigma}\left(\alpha,\beta\right)=V^{*}_{\alpha}A_{N-1}V_{\beta}$. 
\end{proof}

\begin{defn}
\label{def:2-5}We say the boundary data of $K$ at level $N$ are
\textit{shift consistent} if there exist bounded operators $T_{1},\dots,T_{d}$
on $H^{\left(N-1\right)}_{K}$ such that 
\[
\sum_{i}T^{\ast}_{i}T_{i}\le I_{H^{\left(N-1\right)}_{K}},\quad T_{i}V_{\alpha}u=V_{\alpha i}u\quad\left(\alpha\in\Lambda_{N-2}\right)
\]
and, for all $\alpha,\beta\in\Lambda_{N-1}$ and $u,v\in H$, 
\begin{equation}
\left\langle V_{\alpha}u,T_{i}V_{\beta}v\right\rangle _{H^{\left(N\right)}_{K}}=\left\langle u,K\left(\alpha,\beta i\right)v\right\rangle _{H},\label{eq:b2}
\end{equation}
and 
\begin{equation}
\left\langle T_{i}V_{\alpha}u,T_{j}V_{\beta}v\right\rangle _{H^{\left(N\right)}_{K}}=\left\langle u,K\left(\alpha i,\beta j\right)v\right\rangle _{H}.\label{eq:b3}
\end{equation}
In this case one necessarily has 
\begin{equation}
P_{N-2}\Big(\sum^{d}_{i=1}T^{*}_{i}T_{i}\Big)\big|_{H^{\left(N-2\right)}_{K}}=\sum^{d}_{i=1}B^{*}_{i}B_{i}=A_{N-1},\label{eq:b4}
\end{equation}
so the interior density $A_{N-1}$ is the compression of $\sum_{i}T^{*}_{i}T_{i}$. 
\end{defn}

\begin{rem}
At level $N$, the two matching conditions play complementary roles.
Condition \eqref{eq:b2} determines, for each $i$ and $\beta\in\Lambda_{N-1}$,
the vector $T_{i}V_{\beta}u\in H^{(N-1)}_{K}$ uniquely via its pairings
with $H^{(N-1)}_{K}$ (pairings taken in $\langle\cdot,\cdot\rangle_{H^{(N)}_{K}}$),
thereby realizing the interior-boundary entries indexed by $(\alpha,\beta i)$. 

With these vectors fixed, \eqref{eq:b3} requires that the mutual
inner products $\left\langle T_{i}V_{\alpha}u,T_{j}V_{\beta}v\right\rangle _{H^{\left(N\right)}_{K}}$
match the entries indexed by $(\alpha i,\beta j)$, which pins down
the boundary-boundary block. 

In particular, the compression identity \eqref{eq:b4} is necessary
(e.g., from \eqref{eq:b2} or from \eqref{eq:b3} together with \prettyref{cor:2-3})
but does not determine the boundary-boundary block on its own; \eqref{eq:b3}
is needed for that. Together, these matching conditions ensure agreement
with $K$ on $\Lambda_{N}\times\Lambda_{N}$.
\end{rem}

\section{Existence and boundary agreement}\label{sec:3}

The results below address: (i) existence of global p.d. extensions
$\widetilde{K}$ that preserve the interior data and satisfy the shift
inequality, with a Cuntz-Toeplitz realization (\prettyref{thm:3-1});
and (ii) a criterion (shift consistency) for full agreement on the
boundary level $N$ (\prettyref{thm:b8}).
\begin{thm}
\label{thm:3-1}If $K:\Lambda_{N}\times\Lambda_{N}\rightarrow L\left(H\right)$
is p.d. and satisfies $K_{\Sigma}\le K$ on $\Lambda_{N-1}\times\Lambda_{N-1}$,
then $E_{int}\left(K\right)\ne\emptyset$. In particular, there exists
a global p.d. extension $\widetilde{K}$ with (E1)-(E2) and a Cuntz-Toeplitz
realization 
\[
\widetilde{K}\left(\alpha,\beta\right)=W^{\ast}S^{\alpha}P\left(S^{\beta}\right)^{*}W,\quad\alpha,\beta\in\mathbb{F}^{+}_{d},
\]
where $S=\left(S_{1},\dots,S_{d}\right)$ is a row isometry with orthogonal
ranges on a Hilbert space $\mathcal{K}$, $P$ is an orthogonal projection
on $\mathcal{K}$, and $W:H\to\mathcal{K}$ is bounded.
\end{thm}

\begin{proof}
Let $H^{\left(N\right)}_{K}$ be the Kolmogorov space of the truncated
kernel on $\Lambda_{N}$ with canonical vectors $V_{\alpha}u$. Let
$H^{\left(N-1\right)}_{K}$ and $H^{\left(N-2\right)}_{K}$ be as
in \prettyref{eq:2-1}--\prettyref{eq:2-2}.

Let $B_{1},\dots,B_{d}$ be the operators on $H^{\left(N-1\right)}_{K}$
from \prettyref{lem:2-2}. Then 
\[
\sum^{d}_{i=1}B^{*}_{i}B_{i}\le I_{H^{\left(N-1\right)}_{K}},
\]
and $B_{i}V_{\alpha}u=V_{\alpha i}u$ for every $\alpha\in\Lambda_{N-2}$,
$u\in H$, and $i=1,\dots,d$.

By the multivariable dilation theorem for column contractions \cite{MR671311,MR744917,MR972704},
there exist a Hilbert space $\mathcal{K}$, a row isometry $S=(S_{1},\dots,S_{d})$
on $\mathcal{K}$ with orthogonal ranges (i.e., $S^{*}_{i}S_{j}=\delta_{ij}I_{\mathcal{K}}$),
and an isometry 
\[
J:H^{\left(N-1\right)}_{K}\rightarrow\mathcal{K}
\]
such that 
\begin{equation}
S^{*}_{i}J=JB_{i},\quad i=1,\dots,d.\label{eq:dil}
\end{equation}
In particular, $B_{i}=J^{*}S^{*}_{i}J$.

Consequently, for every word $\gamma=i_{1}\cdots i_{k}\in\mathbb{F}^{+}_{d}$,
\begin{equation}
J^{*}\left(S^{\gamma}\right)^{*}J=J^{*}S^{*}_{i_{k}}\cdots S^{*}_{i_{1}}J=B_{i_{k}}\cdots B_{i_{1}}=:B^{\tilde{\gamma}},\label{eq:iter-adj}
\end{equation}
where $\tilde{\gamma}=i_{k}\cdots i_{1}$ is the reversed word. Taking
adjoints in \eqref{eq:iter-adj} gives 
\begin{equation}
J^{*}S^{\gamma}J=\left(J^{*}\left(S^{\gamma}\right)^{*}J\right)^{*}=\left(B^{\tilde{\gamma}}\right)^{*}.\label{eq:iter-nonadj}
\end{equation}
Let $P:=JJ^{*}$ be the orthogonal projection of $\mathcal{K}$ onto
$JH^{\left(N-1\right)}_{K}$.

Define $W:H\to\mathcal{K}$ by 
\[
Wu:=JV_{\emptyset}u,\quad u\in H.
\]
(We do not need $W$ to be an isometry.) Define a kernel $\widetilde{K}$
on $\mathbb{F}^{+}_{d}\times\mathbb{F}^{+}_{d}$ by 
\begin{align}
\widetilde{K}\left(\alpha,\beta\right) & :=W^{*}S^{\alpha}P\left(S^{\beta}\right)^{*}W\label{eq:K-tilde}\\
 & =V^{*}_{\emptyset}J^{*}S^{\alpha}P\left(S^{\beta}\right)^{*}JV_{\emptyset}.\nonumber 
\end{align}
For any finite family $\left\{ \left(\alpha_{k},u_{k}\right)\right\} $
in $\mathbb{F}^{+}_{d}\times H$, 
\[
\sum\nolimits_{j,k}\left\langle u_{j},\widetilde{K}\left(\alpha_{j},\alpha_{k}\right)u_{k}\right\rangle _{H}=\left\Vert P\sum\nolimits_{k}\left(S^{\alpha_{k}}\right)^{*}Wu_{k}\right\Vert ^{2}_{\mathcal{K}}\ge0,
\]
so $\widetilde{K}$ is positive definite.

Fix $\alpha,\beta\in\Lambda_{N-1}$. Using \eqref{eq:iter-nonadj},
\eqref{eq:iter-adj} and $P=JJ^{*}$, we get
\begin{align*}
\widetilde{K}\left(\alpha,\beta\right) & =V^{*}_{\emptyset}J^{*}S^{\alpha}P\left(S^{\beta}\right)^{*}JV_{\emptyset}\\
 & =V^{*}_{\emptyset}\left(J^{*}S^{\alpha}J\right)\left(J^{*}\left(S^{\beta}\right)^{*}J\right)V_{\emptyset}\\
 & =V^{*}_{\emptyset}\left(B^{\tilde{\alpha}}\right)^{*}B^{\tilde{\beta}}V_{\emptyset}.
\end{align*}
We now justify the identification $B^{\tilde{\gamma}}V_{\emptyset}u=V_{\gamma}u$
for all $\gamma\in\Lambda_{N-1}$: if $\gamma=i_{1}\cdots i_{k}$
with $k\le N-1$, then recursively 
\[
B_{i_{1}}V_{\emptyset}u=V_{i_{1}}u,\quad B_{i_{2}}B_{i_{1}}V_{\emptyset}u=V_{i_{1}i_{2}}u,\ \dots,\ B_{i_{k}}\cdots B_{i_{1}}V_{\emptyset}u=V_{i_{1}\cdots i_{k}}u,
\]
and all applications are legal because at each step the current word
length is $\le N-2$ before applying $B_{i_{r+1}}$. Hence $B^{\tilde{\beta}}V_{\emptyset}=V_{\beta}$
and $B^{\tilde{\alpha}}V_{\emptyset}=V_{\alpha}$. Therefore, 
\[
\widetilde{K}\left(\alpha,\beta\right)=V^{*}_{\emptyset}\left(B^{\tilde{\alpha}}\right)^{*}B^{\tilde{\beta}}V_{\emptyset}=V^{*}_{\alpha}V_{\beta}=K\left(\alpha,\beta\right),
\]
i.e., $\widetilde{K}$ agrees with $K$ on $\Lambda_{N-1}\times\Lambda_{N-1}$.

Let 
\[
M:=\overline{span}\left\{ \left(S^{\gamma}\right)^{*}WH:\gamma\in\mathbb{F}^{+}_{d}\right\} \subset\mathcal{K}.
\]
We claim 
\begin{equation}
\sum^{d}_{i=1}S_{i}PS^{*}_{i}\le P\quad\text{on }M.\label{eq:defect-ineq}
\end{equation}
Indeed, fix a finite sum $y=\sum_{r}\left(S^{\gamma_{r}}\right)^{*}Wu_{r}\in M$.
Then 
\[
J^{*}y=\sum_{r}J^{*}\left(S^{\gamma_{r}}\right)^{*}JV_{\emptyset}u_{r}=\sum_{r}B^{\tilde{\gamma}_{r}}V_{\emptyset}u_{r}\in H^{\left(N-1\right)}_{K},
\]
and, for each $i$, 
\[
J^{*}S^{*}_{i}y=\sum_{r}J^{*}\left(S^{\gamma_{r}i}\right)^{*}JV_{\emptyset}u_{r}=\sum_{r}B_{i}B^{\tilde{\gamma}_{r}}V_{\emptyset}u_{r}=B_{i}\left(J^{*}y\right).
\]
Hence 
\[
\sum^{d}_{i=1}\left\Vert J^{*}S^{*}_{i}y\right\Vert ^{2}=\sum^{d}_{i=1}\left\Vert B_{i}\left(J^{*}y\right)\right\Vert ^{2}\le\left\Vert J^{*}y\right\Vert ^{2},
\]
by $\sum_{i}B^{*}_{i}B_{i}\le I_{H^{\left(N-1\right)}_{K}}$. Rewriting
this in $\mathcal{K}$ yields 
\[
\sum^{d}_{i=1}\left\langle y,S_{i}PS^{*}_{i}y\right\rangle _{\mathcal{K}}\le\left\langle y,Py\right\rangle _{\mathcal{K}},
\]
i.e., \eqref{eq:defect-ineq}. Now let $\left\{ \left(\alpha_{j},u_{j}\right)\right\} ^{m}_{j=1}\subset\mathbb{F}^{+}_{d}\times H$
and set 
\[
y:=\sum^{m}_{j=1}\left(S^{\alpha_{j}}\right)^{*}Wu_{j}\in M.
\]
Then, by \eqref{eq:K-tilde} and \eqref{eq:defect-ineq}, 
\begin{align*}
\sum\nolimits^{m}_{j,k=1}\left\langle u_{j},\widetilde{K}_{\Sigma}\left(\alpha_{j},\alpha_{k}\right)u_{k}\right\rangle _{H} & =\left\langle y,\left(\sum\nolimits^{d}_{i=1}S_{i}PS^{*}_{i}\right)y\right\rangle _{\mathcal{K}}\le\left\langle y,Py\right\rangle _{\mathcal{K}}\\[1mm]
 & =\sum\nolimits^{m}_{j,k=1}\left\langle u_{j},\widetilde{K}\left(\alpha_{j},\alpha_{k}\right)u_{k}\right\rangle _{H},
\end{align*}
which is exactly the quadratic form inequality $\widetilde{K}_{\Sigma}\le\widetilde{K}$
on $\mathbb{F}^{+}_{d}\times\mathbb{F}^{+}_{d}$.

Combining the above, $\widetilde{K}$ is a global p.d. kernel satisfying
(E1)-(E2). The displayed formula is its Cuntz-Toeplitz realization.
This shows $E_{int}\left(K\right)\neq\emptyset$ and completes the
proof.
\end{proof}

\begin{thm}
\label{thm:b8}Let $K$ be p.d. on $\Lambda_{N}$ and $K_{\Sigma}\le K$
on $\Lambda_{N-1}\times\Lambda_{N-1}$. Suppose the boundary data
of $K$ at level $N$ are shift consistent (\prettyref{def:2-5}).
Then $E_{bd}\left(K\right)\ne\emptyset$. There exists a global p.d.
kernel $\widetilde{K}$ on $\mathbb{F}^{+}_{d}\times\mathbb{F}^{+}_{d}$
with $\widetilde{K}_{\Sigma}\le\widetilde{K}$ and 
\[
\widetilde{K}\left(\alpha,\beta\right)=K\left(\alpha,\beta\right),\quad\alpha,\beta\in\Lambda_{N}.
\]
\end{thm}

\begin{proof}
Assume there exist bounded operators $T_{1},\dots,T_{d}$ on $H^{\left(N-1\right)}_{K}$
with 
\[
\sum^{d}_{i=1}T^{*}_{i}T_{i}\le I_{H^{\left(N-1\right)}_{K}},\qquad T_{i}V_{\alpha}u=V_{\alpha i}u,\quad\alpha\in\Lambda_{N-2},
\]
and, for all $\alpha,\beta\in\Lambda_{N-1}$, $u,v\in H$, 
\begin{align}
\left\langle V_{\alpha}u,T_{i}V_{\beta}v\right\rangle _{H^{\left(N\right)}_{K}} & =\left\langle u,K\left(\alpha,\beta i\right)v\right\rangle _{H},\label{eq:c12}\\
\left\langle T_{i}V_{\alpha}u,T_{j}V_{\beta}v\right\rangle _{H^{\left(N\right)}_{K}} & =\left\langle u,K\left(\alpha i,\beta j\right)v\right\rangle _{H}.\label{eq:c13}
\end{align}

Apply Frazho-Bunce-Popescu to the row contraction $T=(T_{i})$: there
exist a Hilbert space $\mathcal{K}$, a row isometry $S=(S_{i})$
with orthogonal ranges, and an isometry $J:H^{\left(N-1\right)}_{K}\to\mathcal{K}$
such that $S^{*}_{i}J=JT_{i}$. Moreover, 
\begin{equation}
J^{*}(S^{\gamma})^{*}J=T^{\tilde{\gamma}},\qquad J^{*}S^{\gamma}J=(T^{\tilde{\gamma}})^{*}.\label{eq:intw}
\end{equation}
Set $P:=JJ^{*}$ and $W:=JV_{\emptyset}$, and define the global kernel
\[
\widetilde{K}(\alpha,\beta):=W^{*}S^{\alpha}P(S^{\beta})^{*}W,\qquad\alpha,\beta\in\mathbb{F}^{+}_{d}.
\]
This $\widetilde{K}$ is positive definite by construction. We also
note that $T^{\tilde{\gamma}}V_{\emptyset}=V_{\gamma}$ for all $|\gamma|\le N-1$,
by iterating $T_{i}V_{\eta}=V_{\eta i}$ on $\Lambda_{N-2}$.

\emph{Agreement on $\Lambda_{N}\times\Lambda_{N}$.}

If $|\alpha|,|\beta|\le N-1$: 
\begin{alignat*}{1}
\begin{aligned}\left\langle u,\widetilde{K}\left(\alpha,\beta\right)v\right\rangle _{H} & =\left\langle V_{\emptyset}u,\left(J^{*}S^{\alpha}J\right)(J^{*}\left(S^{\beta}\right)^{*}J)V_{\emptyset}v\right\rangle _{H^{\left(N\right)}_{K}}\\
 & =\left\langle V_{\emptyset}u,\left(T^{\tilde{\alpha}}\right)^{*}T^{\tilde{\beta}}V_{\emptyset}v\right\rangle _{H^{\left(N\right)}_{K}}\\
 & =\left\langle T^{\tilde{\alpha}}V_{\emptyset}u,T^{\tilde{\beta}}V_{\emptyset}v\right\rangle _{H^{\left(N\right)}_{K}}\\
 & =\left\langle V_{\alpha}u,V_{\beta}v\right\rangle _{H^{\left(N\right)}_{K}}\\
 & =\left\langle u,K\left(\alpha,\beta\right)v\right\rangle _{H}.
\end{aligned}
\end{alignat*}

If $|\alpha|\le N-1$ and $\beta=\beta'j$ with $|\beta'|\le N-1$:
\[
\begin{aligned}\left\langle u,\widetilde{K}\left(\alpha,\beta\right)v\right\rangle _{H} & =\left\langle u,W^{*}S^{\alpha}P\left(S^{\beta'j}\right)^{*}Wv\right\rangle _{H^{\left(N\right)}_{K}}\\
 & =\left\langle V_{\emptyset}u,J^{*}S^{\alpha}JJ^{*}\left(S^{\beta'j}\right)^{*}JV_{\emptyset}v\right\rangle _{H^{\left(N\right)}_{K}}\\
 & =\left\langle V_{\emptyset}u,\left(J^{*}S^{\alpha}J\right)\left(J^{*}S^{*}_{j}J\right)\left(J^{*}\left(S^{\beta'}\right)^{*}J\right)V_{\emptyset}v\right\rangle _{H^{\left(N\right)}_{K}}\\
 & =\left\langle V_{\emptyset}u,\left(T^{\tilde{\alpha}}\right)^{*}T_{j}T^{\tilde{\beta'}}V_{\emptyset}v\right\rangle _{H^{\left(N\right)}_{K}}\\
 & =\left\langle T^{\tilde{\alpha}}V_{\emptyset}u,T_{j}T^{\tilde{\beta'}}V_{\emptyset}v\right\rangle _{H^{\left(N\right)}_{K}}\\
 & =\left\langle V_{\alpha}u,T_{j}V_{\beta'}v\right\rangle _{H^{\left(N\right)}_{K}}\\
 & =\left\langle u,K\left(\alpha,\beta'j\right)v\right\rangle _{H}.
\end{aligned}
\]
The case $|\alpha|=N$, $|\beta|\le N-1$ is symmetric.

If $\alpha=\alpha'i$ and $\beta=\beta'j$ with $|\alpha'|,|\beta'|\le N-1$:
\begin{align*}
\left\langle u,\widetilde{K}\left(\alpha'i,\beta'j\right)v\right\rangle _{H} & =\left\langle V_{\emptyset}u,\left(J^{*}S^{\alpha'i}J\right)\left(J^{*}\left(S^{\beta'j}\right)^{*}J\right)V_{\emptyset}v\right\rangle _{H^{\left(N\right)}_{K}}\\
 & =\left\langle V_{\emptyset}u,\left((T^{\widetilde{\alpha'i}})^{*}\right)\left(T^{\widetilde{\beta'j}}\right)V_{\emptyset}v\right\rangle _{H^{\left(N\right)}_{K}}\\
 & =\left\langle V_{\emptyset}u,(T^{\tilde{\alpha'}})^{*}T^{*}_{i}T_{j}T^{\tilde{\beta'}}V_{\emptyset}v\right\rangle _{H^{\left(N\right)}_{K}}\\
 & =\left\langle T^{\tilde{\alpha'}}V_{\emptyset}u,T^{*}_{i}T_{j}T^{\tilde{\beta'}}V_{\emptyset}v\right\rangle _{H^{\left(N\right)}_{K}}\\
 & =\left\langle V_{\alpha'}u,T^{*}_{i}T_{j}V_{\beta'}v\right\rangle _{H^{\left(N\right)}_{K}}\\
 & =\left\langle T_{i}V_{\alpha'}u,T_{j}V_{\beta'}v\right\rangle _{H^{\left(N\right)}_{K}}\\
 & =\left\langle u,K\left(\alpha'i,\beta'j\right)v\right\rangle _{H}.
\end{align*}

\noindent Thus $\widetilde{K}=K$ on $\Lambda_{N}\times\Lambda_{N}$.

\emph{Shift inequality.} Let $M:=\overline{\mathrm{span}}\{(S^{\gamma})^{*}WH:\gamma\in\mathbb{F}^{+}_{d}\}$.
For $y=\sum_{r}(S^{\gamma_{r}})^{*}Wu_{r}\in M$, \eqref{eq:intw}
gives $J^{*}y=\sum_{r}T^{\tilde{\gamma_{r}}}V_{\emptyset}u_{r}$ and
$J^{*}S^{*}_{i}y=T_{i}(J^{*}y)$, hence $\sum_{i}\|J^{*}S^{*}_{i}y\|^{2}=\sum_{i}\|T_{i}(J^{*}y)\|^{2}\le\|J^{*}y\|^{2}$.
Equivalently, $\sum_{i}\langle y,S_{i}PS^{*}_{i}y\rangle\le\langle y,Py\rangle$
on $M$, which is exactly $\widetilde{K}_{\Sigma}\le\widetilde{K}$.
\end{proof}

\begin{rem}
\label{rem:c3}We note that the classical truncated Hausdorff moment
problem (THMP) on $\left[0,1\right]$ fits our setting. Here, let
$d=1$, so $\mathbb{F}^{+}_{1}\simeq\mathbb{N}$. For a finite positive
Borel measure $\mu$ on $\left[0,1\right]$, consider the kernel 
\begin{equation}
K(m,n)=\int^{1}_{0}x^{m+n}d\mu(x),\quad m,n\in\mathbb{N}.\label{eq:d1}
\end{equation}

In the classical THMP one is given finitely many moments $s_{0},\dots,s_{M}$
and asks for a measure $\mu$ on $[0,1]$ with $s_{k}=\int^{1}_{0}x^{k}d\mu(x)$.
This imposes (i) a Hankel constraint $K(m,n)=s_{m+n}$ and (ii) the
Hausdorff feasibility (complete monotonicity) conditions on $(s_{k})$.
Whenever these hold, the associated kernel \eqref{eq:d1} is positive
definite. Moreover, with $K_{\Sigma}\left(m,n\right):=K\left(m+1,n+1\right)$,
one has $K_{\Sigma}\le K$ (see \prettyref{lem:c4}). Hence, THMP
feasibility $\Longrightarrow$ the hypotheses used in our extension
theorems (for $d=1$).

The converse fails in general: our assumptions (p.d. on a truncation
plus $K_{\Sigma}\le K$ on the interior) do not enforce Hankel structure
or the Hausdorff inequalities, so a dilation-based global extension
need not arise from a measure. When, in addition, one enforces the
Hankel form $\widetilde{K}(m,n)=s_{m+n}$ and the truncated Hausdorff
(complete-monotonicity) conditions, the two formulations coincide:
the model operator can be chosen as a self-adjoint contraction (multiplication
by $x$ on $L^{2}\left(\mu\right)$), and $\widetilde{K}(m,n)=\int^{1}_{0}x^{m+n}d\mu(x)$. 
\end{rem}

\begin{lem}
\label{lem:c4}Let $K$ be as in \eqref{eq:d1}. Then $K\geq0$ and
$K_{\Sigma}\le K$. Moreover, for every $N$, the truncations satisfy
$K|_{\Lambda_{N}\times\Lambda_{N}}\geq0$, and $K_{\Sigma}\le K$
on the interior block $\Lambda_{N-1}$. (Recall that $\Lambda_{N}=\left\{ 0,\dots,N\right\} $
for every $N\in\mathbb{N}$.)
\end{lem}

\begin{proof}
For any finitely supported complex scalars $\left(c_{m}\right)$,
\[
\sum\nolimits_{m,n\geq0}\overline{c_{m}}c_{n}K\left(m,n\right)=\int^{1}_{0}\left|\sum\nolimits_{m\geq0}c_{m}x^{m}\right|^{2}d\mu\left(x\right)\geq0.
\]
Also, $K_{\Sigma}(m,n)=\int^{1}_{0}x^{m+n+2}d\mu(x)$, hence 
\begin{multline*}
\sum\nolimits_{m,n\geq0}\overline{c_{m}}c_{n}\left(K\left(m,n\right)-K_{\Sigma}\left(m,n\right)\right)=\\
\int^{1}_{0}\left|\sum\nolimits_{m\geq0}c_{m}x^{m}\right|^{2}\left(1-x^{2}\right)d\mu\left(x\right)\geq0,
\end{multline*}
so $K_{\Sigma}\leq K$. Restricting sums to $0\leq m,n\leq N-1$ gives
the truncated interior inequality. 
\end{proof}

\section{Examples}\label{sec:4}

Let $K:\Lambda_{N}\times\Lambda_{N}\rightarrow L\left(H\right)$ be
a p.d. kernel. We consider the scalar case $H=\mathbb{C}$ with $d=2$.
The index sets are $\Lambda_{1}=\left\{ 0,1,2\right\} $, $\Lambda_{2}=\left\{ 0,1,2,11,12,21,22\right\} $,
$\partial\Lambda_{2}=\left\{ 11,12,21,22\right\} $, and $\Lambda_{2}=\Lambda_{1}\cup\partial\Lambda_{2}$.
\begin{example}[Shift-consistency boundary $\Longrightarrow$ \prettyref{thm:b8}
applies]
\label{exa:d1} Work in $\mathbb{C}^{3}$ (the Kolmogorov space)
with standard orthonormal basis $\left\{ e_{0},e_{1},e_{2}\right\} $.
Set 
\[
V_{0}=e_{0},\quad V_{1}=\frac{1}{2}e_{1},\quad V_{2}=\frac{1}{2}e_{2}.
\]
Then $K\left(\alpha,\beta\right)=\left\langle V_{\alpha},V_{\beta}\right\rangle $
on $\Lambda_{1}$ gives 
\[
K|_{\Lambda_{1}\times\Lambda_{1}}=diag\left(1,1/4,1/4\right)\quad\left(\text{order }0,1,2\right).
\]

Define $T_{1}$, $T_{2}:\mathbb{C}^{3}\rightarrow\mathbb{C}^{3}$
by 
\begin{align*}
T_{1}e_{0} & =\frac{1}{2}e_{1},\quad T_{1}e_{1}=0,\quad T_{1}e_{2}=0;\\
T_{2}e_{0} & =\frac{1}{2}e_{2},\quad T_{2}e_{1}=0,\quad T_{2}e_{2}=0.
\end{align*}
Set 
\[
A:=T^{*}_{1}T_{1}+T^{*}_{2}T_{2}=diag\left(1/2,0,0\right)\leq I_{\mathbb{C}^{3}},
\]
so $\left(T_{1},T_{2}\right)$ is a row contraction on the level-1
space. 

Using $V_{\alpha i}:=T_{i}V_{\alpha}$ for $\left|\alpha\right|\leq1$,
we get 
\begin{align*}
V_{11} & =T_{1}V_{1}=0,\quad V_{12}=T_{2}V_{1}=0,\\
V_{21} & =T_{1}V_{2}=0,\quad V_{22}=T_{2}V_{2}=0,
\end{align*}
and extend $K$ to $\Lambda_{2}$. 

For $\alpha,\beta\in\Lambda_{1}$, 
\[
K_{\Sigma}\left(\alpha,\beta\right)=\sum^{2}_{i=1}K\left(\alpha i,\beta i\right)=\sum^{2}_{i=1}\left\langle T_{i}V_{\alpha},T_{i}V_{\beta}\right\rangle =\left\langle AV_{\alpha},V_{\beta}\right\rangle .
\]
Hence, for any scalars $c_{\alpha}$ supported on $\Lambda_{1}$,
\[
\sum_{\alpha,\beta}\overline{c_{\alpha}}c_{\beta}\left(K\left(\alpha,\beta\right)-K_{\Sigma}\left(\alpha,\beta\right)\right)=\left\Vert \left(I-A\right)^{1/2}\sum c_{\alpha}V_{\alpha}\right\Vert ^{2}\geq0,
\]
so $K_{\Sigma}\leq K$ on $\Lambda_{1}$. Concretely, 
\[
K|_{\Lambda_{1}\times\Lambda_{1}}=diag\left(1,1/4,1/4\right),\quad K_{\Sigma}|_{\Lambda_{1}\times\Lambda_{1}}=diag\left(1/2,0,0\right),
\]
and so the difference $diag\left(1/2,1/4,1/4\right)$ is p.d. 

By construction, 
\[
K\left(\alpha i,\beta j\right)=\left\langle T_{i}V_{\alpha},T_{j}V_{\beta}\right\rangle ,\quad K\left(\alpha,\beta i\right)=\left\langle V_{\alpha},T_{i}V_{\beta}\right\rangle ,
\]
for all $\alpha,\beta\in\Lambda_{1}$. This is the shift-consistency
condition. Hence the boundary is shift consistent. 

\textbf{Conclusion.} Since $K_{\Sigma}\leq K$ on $\Lambda_{1}$,
\prettyref{thm:3-1} yields a global positive definite extension preserving
the interior; and since the boundary is shift-consistent, \prettyref{thm:b8}
gives a global extension that also agrees with $K$ on $\Lambda_{2}\times\Lambda_{2}$. 
\end{example}

\begin{rem}
Let $H_{K}$ denote the Kolmogorov space for the truncated kernel
$K$ (defined on $\Lambda_{N}\times\Lambda_{N}$), and set
\[
H^{\left(m\right)}_{K}:=span\{V_{\gamma}:|\gamma|\le m\}\subset H_{K},\quad m=1,2.
\]

In \prettyref{exa:d1} we have $H^{\left(1\right)}_{K}=H^{\left(2\right)}_{K}=H_{K}=\mathbb{C}^{3}$,
because the length-2 vectors add no new directions. In this special
situation, a global p.d. extension that preserves the interior already
follows from the $\Lambda_{1}$ data and the inequality $K_{\Sigma}\le K$
(\prettyref{thm:3-1}); the level-2 block plays no role for mere existence.
The additional shift-consistency requirement is only needed if one
further demands boundary agreement on $\Lambda_{2}\times\Lambda_{2}$
(\prettyref{thm:b8} ).

By contrast, when $H^{\left(2\right)}_{K}\supsetneq H^{\left(1\right)}_{K}$
the boundary introduces genuinely new directions, and the shift-consistency
condition becomes a substantive compatibility constraint for achieving
boundary agreement.
\end{rem}

\begin{example}[Not shift-consistent $\Longrightarrow$ \prettyref{thm:b8} not applicable]
\label{exa:d2}  Keep the scalar case $H=\mathbb{C}$, $d=2$, with
index sets $\Lambda_{1}$, $\partial\Lambda_{2}$, $\Lambda_{2}=\Lambda_{1}\cup\partial\Lambda_{2}$
as above.

Work in $\mathbb{C}^{4}$ with orthonormal basis $\{e_{0},e_{1},e_{2},e_{3}\}$.
Set 
\[
V_{0}=e_{0},\quad V_{1}=\tfrac{1}{2}e_{1},\quad V_{2}=\tfrac{1}{2}e_{2},
\]
so 
\[
K|_{\Lambda_{1}\times\Lambda_{1}}=diag\left(1,1/4,1/4\right)\quad\text{(order }(0,1,2)\text{)}.
\]

Define 
\[
V_{11}=0,\quad V_{21}=0,\quad V_{22}=0,\quad V_{12}=\tfrac{1}{4}e_{3}.
\]
Extend $K$ to $\Lambda_{2}$ by the Gram rule $K(\gamma,\delta)=\langle V_{\gamma},V_{\delta}\rangle$.

Then we have 
\[
\begin{aligned} & K_{\Sigma}(0,0)=K(1,1)+K(2,2)=\tfrac{1}{4}+\tfrac{1}{4}=\tfrac{1}{2},\\
 & K_{\Sigma}(1,1)=K(11,11)+K(12,12)=0+\|V_{12}\|^{2}=\tfrac{1}{16},\\
 & K_{\Sigma}(2,2)=K(21,21)+K(22,22)=0+0=0,\\
 & K_{\Sigma}(1,2)=K(11,12)+K(12,22)=0,\qquad K_{\Sigma}(0,1)=K(1,11)+K(2,12)=0,\\
 & K_{\Sigma}(0,2)=K(1,21)+K(2,22)=0.
\end{aligned}
\]
Hence (order $(0,1,2)$) 
\[
K_{\Sigma}|_{\Lambda_{1}\times\Lambda_{1}}=diag\left(1/2,1/16,0\right),\quad K|_{\Lambda_{1}\times\Lambda_{1}}-K_{\Sigma}|_{\Lambda_{1}\times\Lambda_{1}}=diag\left(1/2,3/16,1/4\right)\geq0.
\]
Thus $K_{\Sigma}\le K$ on $\Lambda_{1}$, and by \prettyref{thm:3-1}
a global p.d. extension exists that preserves the interior.

Shift-consistency (\prettyref{def:2-5}) would require operators $T_{1},T_{2}$
on the level-1 space 
\[
H^{\left(1\right)}_{K}=span\left\{ V_{0},V_{1},V_{2}\right\} =span\left\{ e_{0},e_{1},e_{2}\right\} .
\]
Since $V_{12}=\frac{1}{4}e_{3}$ is orthogonal to $H^{\left(1\right)}_{K}$,
condition \eqref{eq:b2}, with $\beta=1$ and $i=2$, gives 
\[
\left\langle V_{\alpha},T_{2}V_{1}\right\rangle =K\left(\alpha,12\right)=\left\langle V_{\alpha},V_{12}\right\rangle =0
\]
for every $\alpha\in\Lambda_{1}$. Hence 
\[
T_{2}V_{1}=0.
\]
On the other hand, condition \eqref{eq:b3}, with $\alpha=\beta=1$
and $i=j=2$, gives 
\[
\left\Vert T_{2}V_{1}\right\Vert ^{2}=K\left(12,12\right)=\frac{1}{16},
\]
a contradiction. Therefore the level-2 boundary is not shift-consistent,
and \prettyref{thm:b8} is not applicable (we make no claim about
the existence or nonexistence of a boundary-preserving extension).

Note that, with these choices 
\[
H^{\left(1\right)}_{K}=span\{e_{0},e_{1},e_{2}\}\subsetneq H^{\left(2\right)}_{K}=span\{e_{0},e_{1},e_{2},e_{3}\},
\]
so the level-2 data introduce a new direction. 
\end{example}

\bibliographystyle{amsalpha}
\bibliography{ref}

\end{document}